\documentclass[12pt]{amsart}

\usepackage{amsthm, amsmath}
\usepackage{amssymb}
\usepackage[all]{xy}
\usepackage{mathrsfs}
\usepackage[latin1]{inputenc}
\usepackage{graphicx}

\numberwithin{equation}{subsection}

\DeclareMathOperator{\Flag}{Flag}

\DeclareMathOperator{\Gal}{Gal}

\theoremstyle{plain}

\newtheorem{theorem}{Theorem}[subsection]

\newtheorem*{thrm1}{Theorem A}

\newtheorem*{thrm2}{Theorem B}

\newtheorem{lemma}[theorem]{Lemma}

\newtheorem{cor}[theorem]{Corollary}

\newtheorem{defn}[theorem]{Definition}
\newtheorem{prop}[theorem]{Proposition}

\theoremstyle{definition}
\newtheorem{rem}[theorem]{Remark}

\newcommand\isomarrow{\stackrel{\cong}{\longrightarrow}}

\def\ge{\geqslant}

\def\a{\alpha}
\def\b{\beta}

\def\G{\Gamma}
\def\d{\delta}

\def\e{\epsilon}

\def\s{\sigma}

\def\th{\theta}
\def\k{\kappa}
\def\l{\lambda}

\def\i{^{-1}}
\def\tSS{\tilde{\mathbb S}}
\def\SS{\mathbb S}
\def\ZZ{\mathbb Z}
\def\NN{\mathbb N}
\def\QQ{\mathbb Q}
\def\RR{\mathbb R}

\def\ca{\mathcal A}

\def\co{\mathcal O}

\def\tW{\tilde W}

\def\ba{\mathbf a}

\usepackage{amssymb} 
\usepackage{latexsym} 
\usepackage{amsfonts} 
\usepackage{amsmath} 
\usepackage{eucal} 
\usepackage{bm} 
\usepackage{bbm} 
\usepackage{graphicx} 
\usepackage[english]{varioref} 
\usepackage[nice]{nicefrac} 
\usepackage[all]{xy}

\newcommand{\kk}{\Bbbk}

\newcommand{\Gad}{G_{\rm ad}}
\newcommand{\Gsc}{G_{\rm sc}}

\def\<{\langle}
\def\>{\rangle}

\begin{document}

\title[Affine Deligne-Lusztig varieties]{$P$-alcoves and nonemptiness of affine Deligne-Lusztig varieties}

\author[U. G\"{o}rtz]{Ulrich G\"{o}rtz}
\address{Ulrich G\"{o}rtz\\Institut f\"ur Experimentelle Mathematik\\Universit\"at Duisburg-Essen\\Ellernstr.~29\\45326 Essen\\Germany}
\email{ulrich.goertz@uni-due.de}
\thanks{G\"{o}rtz was partially supported by the Sonderforschungsbereich TR 45 ``Periods, Moduli spaces and Arithmetic of Algebraic Varieties'' of the Deutsche Forschungsgemeinschaft.}

\author[X. He]{Xuhua He}
\address{Xuhua He, Department of Mathematics, The Hong Kong University of Science and Technology, Clear Water Bay, Kowloon, Hong Kong}
\thanks{Xuhua He was partially supported by HKRGC grant 602011.}
\email{maxhhe@ust.hk}

\author[S. Nie]{Sian Nie}
\address{Sian Nie, IHES,  Le Brois-Marie, 35 Rue de Chartres, 91440 Bures-sur-Yvette, France}
\email{niesian@ihes.fr}

\begin{abstract}
We study affine Deligne-Lusztig varieties in the affine flag manifold of an algebraic group, and in particular the question, which affine Deligne-Lusztig varieties are non-empty. Under mild assumptions on the group, we provide a complete answer to this question in terms of the underlying affine root system. In particular, this proves the corresponding conjecture for split groups stated in \cite{GHKR2}. The question of non-emptiness of affine Deligne-Lusztig varieties is closely related to the relationship between certain natural stratifications of moduli spaces of abelian varieties in positive characteristic.
\end{abstract}

\maketitle

\section{Introduction}

\subsection{}
Affine Deligne-Lusztig varieties (see below for the definition) are the analogues of Deligne-Lusztig varieties in the context of an affine root system, and hence are natural objects which deserve to be studied in their own interest. Furthermore, results about them have direct applications to certain questions in arithmetic geometry, specifically to moduli spaces of $p$-divisible groups and reductions of Shimura varieties. More concretely, if $\mathcal M$ is a Rapoport-Zink space, then $\mathcal M(\kk)$ can be identified by Dieudonn\'e theory with a (mixed-characteristic) affine Deligne-Lusztig variety. In this case, the formal scheme $\mathcal M$ provides a scheme structure. See \cite{GHKR1} 5.10 for further information on this connection.

\subsection{}
Let $\mathbb F_q$ be the finite field with $q$ elements. Let $\kk$ be an algebraic closure of $\mathbb F_q$. Consider one of the following two cases:
\begin{itemize}
\item
Mixed characteristic case. Let $F/\mathbb Q_p$ be a finite field extension with residue class field $\mathbb F_q$, and let $L$ be the completion of the maximal unramified extension of $F$. Denote by $\varepsilon$ a uniformizer of $F$.
\item
Equal characteristic case. Let $F= \mathbb F_q( (\e))$, the field of Laurent series over $\mathbb F_q$, and $L:=\kk( (\e))$, the field of Laurent series over $\kk$. As in the previous case, $L$ is the completion of the maximal unramified extension of $F$.
\end{itemize}

Let $G$ be a connected semisimple group over $F$ which splits over a tamely ramified extension of $F$.  Let $\s$ be the Frobenius automorphism of $L/F$. We also denote the induced automorphism on $G(L)$ by $\s$.

We fix a $\s$-invariant Iwahori subgroup $I\subset G(L)$. In the equal characteristic case we can view $G(L)/I$ as the $\kk$-points of an ind-projective ind-scheme $\Flag$ over $\kk$, the affine flag variety for $G$, see~\cite{pappas-rapoport:twisted}. The $I$-double cosets in $G(L)$ are parameterized by the extended affine Weyl group $\tW$. The automorphism on $\tW$ induced by $\s$ is denoted by $\d\colon \tW\to \tW$. Furthermore we denote by $\mathbb S\subseteq \tW$ the set of simple affine reflections.

Following Rapoport \cite{rapoport:guide}, we define:

\begin{defn}
Let $x\in \tW$, and $b\in G(L)$. The affine Deligne-Lusztig variety attached to $x$ and $b$ is the subset
\[
X_x(b) = \{ gI \in G(L)/I;\ g^{-1}b\sigma(g) \in I xI \}.
\]
\end{defn}

In the equal characteristic case, it is not hard to see that there exists a unique locally closed $X_x(b)\subset \Flag$ whose set of $\kk$-valued points is the subset $X_x(b)\subseteq G(L)/I$ defined above. Moreover, $X_x(b)$ is a finite-dimensional $\kk$-scheme, locally of finite type over $\kk$ 
(but not in general of finite type: depending on $b$, $X_x(b)$ may have infinitely many irreducible components). In the mixed characteristic case, the term ``variety'' is not really justified. More precisely one should speak about affine Deligne-Lusztig sets.

As experience and partial results show, many basic properties of affine Deligne-Lusztig varieties such as non-emptiness and dimension depend only on the underlying combinatorial structure of the (affine) root system, and therefore coincide in the mixed characteristic and equal characteristic cases.

For the remainder of the introduction, we fix a basic element $b \in G(L)$, i.e., an element whose Newton vector is central, or equivalently, whose $\sigma$-conjugacy class can be represented by a length zero element of the extended affine Weyl group. Compare \cite{kottwitz-isoI} and Section~\ref{sec:sigmaconj}.

So far, the main questions that have been studied are
\begin{enumerate}
\item
For which $x$ is $X_x(b)\ne\emptyset$?
\item
If $X_x(b)\ne\emptyset$ and $X_x(b)$ carries a scheme structure, what is $\dim X_x(b)$?
\end{enumerate}

Until recently, most of the results have been established only for split groups.
For tamely ramified quasi-split groups, we refer to \cite{He-GeometryOfADLV} Section 12 for question 2, at least in the equal-characteristic case.

In this paper, we focus on Question 1 above and give a complete answer to this question. 

We first show that it suffices to consider quasi-split, semisimple groups of adjoint type (see Sections \ref{sec:reduce-adjoint}, \ref{sec:reduce-quasisplit} for an explanation how to reduce to this case). For such groups, the answer is given in terms of the affine root system and the affine Weyl group of $G$ and uses the notion of $(J, w, \d)$-alcove (see Section~\ref{sec:Jwd}), a generalization of the notion of $P$-alcove introduced in \cite{GHKR2} for split groups.

The definition of a $(J, w, \d)$-alcove is a little bit technical, so we do not state it in this introduction. See Section~\ref{sec:Jwd} for the details. Roughly, two conditions must be met by an alcove $x\mathbf a$, $x\in \tW$, to be a $(J, w, \d)$-alcove: $x$ must satisfy a restriction on its finite part, and the alcove must lie in a certain region of the apartment, which is essentially a union of certain finite Weyl chambers. See \cite{GHKR2}~Section 3 for a visualization.

We denote by $\Gamma_F$ the absolute Galois group of $F$ and by $\kappa_G\colon G(L)\to \pi_1(G)_{\Gamma_F}$ the Kottwitz map; see \cite{kottwitz-isoI}, \cite{rapoport-richartz}, and Section~\ref{sec:PropertyStar}. Note that $\kappa_G$ also gives rise to maps with source $\tW$ and source $B(G)$. Likewise, for a Levi subgroup $M$, we denote by $\kappa_M$ the corresponding Kottwitz map. 

\begin{thrm1}[Corollary~\ref{ConsequBasic}, Theorem~\ref{thm-nonemptiness}]
Let $b\in G(L)$ be a basic element, and let $x\in\tW$. Then $X_x(b) = \emptyset$ if and only if there exists a pair $(J, w)$ such that $x \mathbf a$ is a $(J, w, \d)$-alcove and $$\kappa_{M_J}(w \i x \d(w))\not\in \kappa_{M_J}([b]\cap M_J(L)).$$  
\end{thrm1}

We say that the Dynkin diagram of $G$ is $\d$-connected if it can't be written as a union of two proper $\d$-stable subdiagrams that are not connected to each other.  We say that an element $w \in \tW$ lies in the shrunken Weyl chambers if $x \mathbf a$ does not lie in the same strip as the base alcove $\mathbf a$ with respect to any root direction (cf.~Prop.~\ref{reuman3}).

In this case, we have a more explicit description of the nonemptiness behavior of $X_x(b)$. The answer is given in terms of the map $\eta_\d$ from $\tW$ to the finite Weyl group $W$ defined in Section 3.5. 



\begin{thrm2}[Proposition~\ref{reuman3}, Proposition~\ref{reuman4}]
Assume that the Dynkin diagram of $G$ is $\d$-connected. Let $x\in\tW$ lie in the shrunken Weyl chambers. Let $b\in G(L)$ be a basic element. Then $X_x(b) \ne \emptyset$ if and only if $\kappa_G(x) = \kappa_G(b)$ and $\eta_\d(x) \in W - \bigcup_{J\subsetneq \SS, \d(J)=J} W_J$. 
\end{thrm2}

Let us give an overview of the paper. In Section~\ref{SEC:preliminaries} we collect some preliminaries and reduce to the case that $G$ is quasi-split and semisimple of adjoint type. In Section~\ref{SEC:emptiness} we prove, imitating the proof given in \cite{GHKR2} in the split case, the direction of Theorem A claiming emptiness. In the final Section~\ref{SEC:nonemptiness} we prove the non-emptiness statement of the theorem by employing the ``reduction method'' of Deligne and Lusztig. We show that the notion of $(J, w, \d)$-alcove is compatible with this reduction. Using some interesting combinatorial properties of affine Weyl groups established by the second-named and third-named authors~\cite{HN}, we are able to reduce the question to the case of $X_x(b)$, where $x$ is of minimal length in its $\d$-conjugacy class. This case can be handled directly using the explicit description of minimal length elements in \cite{HN}.

\emph{Acknowledgments.} We thank Timo Richarz for his help with the theory of Iwahori-Weyl groups for non-split groups.
We also thank Allen Moy and Xinwen Zhu for helpful discussions. 

\section{Preliminaries}
\label{SEC:preliminaries}

\subsection{Notation}\label{sec:notation}






Let $S\subset G$ be a maximal $L$-split torus defined over $F$. The centralizer $T$ of $S$ in $G$ is a maximal torus, because over $L$, $G$ is quasi-split. The Frobenius automorphism $\s$ of $L/F$ acts on the Iwahori-Weyl group
\[
\tW = N_S(L)/T(L)_1.
\]
Here $N_S$ denotes the normalizer of $S$ in $G$, and $T(L)_1$ denotes the unique parahoric subgroup of $T(L)$. For $w \in \tW$, we choose a representative in $N_S(L)$ and also write it as $w$. 

We denote by $\mathcal A$ the apartment of $G_L$ corresponding to $S$. 
We fix a $\sigma$-invariant alcove $\mathfrak a$ in $\ca$, and by $I\subseteq G(L)$ the Iwahori subgroup corresponding to $\mathfrak a$ over $L$.

\subsubsection{The affine Weyl group}

Denote by $G_1\subset G(L)$ the subgroup generated by all parahoric subgroups. We denote by
\[
W_a:=(N_S(L)\cap G_1)/(N_S(L)\cap I)
\]
the affine Weyl group.

The affine Weyl group acts simply transitively on the set of alcoves in $\mathcal A$, and our choice of base alcove gives rise to a length function and the Bruhat order on $W_a$. As usual, the length of an alcove is the number of ``affine root hyperplanes'' in the apartment separating the alcove from the base alcove.

\subsubsection{Semi-direct product representations of the Iwahori-Weyl group}

Denote by $\Gamma$ the absolute Galois group $\Gal(\overline{L}/L)$ of $L$. We can identify $\Gamma$ with the inertia subgroup of the absolute Galois group $\Gamma_F$ of $F$. By a subscript $\bullet{}_\Gamma$ we denote $\Gamma$-coinvariants.

Denote by $W = N_S(L)/T(L)$ the (relative, finite) Weyl group of $G$ with respect to $S$. 

We use the following important short exact sequences:
\begin{equation}\label{tW-first-sequence}
0 \to X_*(T)_\Gamma \to \tW \to W \to 1,
\end{equation}
where the map $\tW\to W$ is the natural projection. Its kernel is $T(L)/T(L)_1$ which can be identified with $X_*(T)_\Gamma$, see~\cite{pappas-rapoport:twisted}, Section 5. This short exact sequence splits, and we obtain $\tW = X_*(T)_\Gamma\rtimes W$. See \cite{haines-rapoport}, Proposition 13.

On the other hand, the affine Weyl group naturally embeds into $\tW$, and we have an exact sequence
\[
1 \to W_a \to \tW \to X^*(Z(\widehat{G})^\Gamma) \to 0.
\]
We can identify $X^*(Z(\widehat{G})^\Gamma)$ with the stabilizer of the base alcove $\mathfrak a$ in $\tW$. This shows that $\tW = W_a \rtimes X^*(Z(\widehat{G})^\Gamma)$. See \cite{haines-rapoport}, Lemma 14. Setting $\ell(x)=0$ for $x\in X^*(Z(\widehat{G})^\Gamma)$, we extend the length function to $\tW$.

At the same time, we can view $W_a$ as the Iwahori-Weyl group of the simply connected cover $G_{\rm sc}$ of the derived group $G_{\rm der}$ of $G$. Denoting by $T_{\rm sc}\subset G_{\rm sc}$ the maximal torus given by the choice of $T$, we obtain a semi-direct product decomposition
\[
W_a = X_*(T_{\rm sc})_\Gamma \rtimes W.
\]
We can identify $W_a$ with the group generated by the reflections with respect to the walls of $\mathfrak a$.

\subsubsection{Affine flag varieties}

The structure theory for $G(L)$ established by Bruhat and Tits gives the Iwahori-Bruhat decomposition
\[
G(L) = \bigcup_{w\in \tW} I w I,\qquad
G(L)/I = \bigcup_{w\in \tW} I w I/I,
\]
where both unions are disjoint.

\subsection{Reduction to adjoint groups}
\label{sec:reduce-adjoint}

Let $G$ be a connected semisimple group over $F$, and let $\Gad$ be the corresponding group of adjoint type, i.e., the quotient of $G$ by its center. The buildings of $G$ and $\Gad$ coincide, so that the choice of an alcove $\mathbf a$ in the building of $G$ determines an alcove, and hence an Iwahori group of $\Gad$. We first consider the more complicated case of equal characteristic.

Denote by $\Flag$ and $\Flag_{\rm ad}$ the corresponding affine flag varieties for $G$ and $\Gad$.

\begin{prop}
Assume that $\mathop{\rm char} \kk$ does not divide the order of $\pi_1(\Gad)$.
\begin{enumerate}
\item
The homomorphism $G\to \Gad$ induces an immersion
\[
\Flag \to \Flag_{\rm ad}.
\]
\item
Let $\l\in \pi_0(\Flag) = \pi_1(G)_\Gamma$, denote by $\l_{\rm ad}$ its image under the injective map $\pi_0(\Flag) \to \pi_0(\Flag_{\rm ad})$, and denote by $\Flag_\l$ and $\Flag_{{\rm ad}, \l_{\rm ad}}$ the corresponding connected components. Then the above immersion induces an isomorphism
\[
\Flag_\l \isomarrow \Flag_{{\rm ad}, \l_{\rm ad}}.
\]
\end{enumerate}
\end{prop}

\begin{proof}
Denote by $\Gsc$ the simply connected cover of $G$, and by $\Flag_{\rm sc}$ its affine flag variety (attached to the Iwahori of $\Gsc$ given by $\mathbf a$). It is proved in \cite{pappas-rapoport:twisted} 6.a that there are natural maps
\[
\Flag_{\rm sc} \to \Flag \to \Flag_{\rm ad}
\]
and that
\[
\Flag_{\rm sc} \to \Flag,\quad\text{and}\quad\Flag_{\rm sc} \to \Flag_{\rm ad}
\]
are immersions which identify $\Flag_{\rm sc}$ with the neutral connected component of $\Flag$ and of $\Flag_{\rm ad}$.
Now let $\l\in \pi_0(\Flag) = \pi_1(G)_\Gamma$ (cf. \cite{pappas-rapoport:twisted} Theorem 5.1). Since $\pi_0(\Flag) = \pi_0(LG)$, we can find a representative $g\in LG(\kk)$ of $\lambda$. Left multiplication identifies the neutral connected component $\Flag_0$ with $\Flag_\l$, and likewise the image of $g$ in $L\Gad(\kk)$ identifies $\Flag_{{\rm ad}, 0}$ with $\Flag_{{\rm ad}, \l_{\rm ad}}$. This proves the proposition.
\end{proof}

Choosing maximal tori in $G$ and $\Gad$ compatibly, we obtain a map $x\mapsto x_{\rm ad}$ between the corresponding extended affine Weyl groups. For $b\in G(L)$, we denote by $b_{\rm ad}$ its image in $\Gad(L)$. Finally, for an affine Deligne-Lusztig variety $X_x(b)$ and $\l\in \pi_0(\Flag)$, we denote by $X_x(b)_\l$ the intersection $X_x(b)\cap \Flag_\l$, and likewise for $\Gad$.

In the mixed characteristic case, an analogous set-theoretic statement is true without any assumption on the order of $\pi_1(\Gad)$. The notion of \emph{connected component} should be replaced by \emph{fiber of the Kott\-witz homomorphism $G(L) \to \pi_1(G)_\Gamma$}. This can be shown along the same lines as above.

The proposition immediately implies the following corollary. Compare the discussion before Prop.~5.9.2 in \cite{GHKR1} for an analogous statement for split groups and affine Grassmannians.

\begin{cor}
\begin{enumerate}
\item
(Equal characteristic case)
Assume that $\mathop{\rm char} \kk$ does not divide the order of $\pi_1(\Gad)$.
Let $b\in G(L)$, $x\in \tW$, and $\l\in \pi_0(\Flag)$. Then the isomorphism $\Flag_\l \isomarrow \Flag_{{\rm ad}, \l_{\rm ad}}$ induces an isomorphism
\[
X_x(b)_\l = X_{x_{\rm ad}}(b_{\rm ad})_{\l_{\rm ad}}.
\]
\item
(Mixed characteristic case)
Let $b\in G(L)$, $x\in \tW$, and $\l\in \pi_0(\Flag)$. We have a bijection
\[
X_x(b)_\l = X_{x_{\rm ad}}(b_{\rm ad})_{\l_{\rm ad}}.
\]
\end{enumerate}
\end{cor}

\subsection{Reduction to the quasi-split case}\label{quasi-split}
\label{sec:reduce-quasisplit}

Let $H$ be a connected semisimple group over $F$ of adjoint type and $G$ be its quasi-split inner form. As before, we denote by $\tW$ the Iwahori-Weyl group of $G$ (over $L$). The inner forms of $G$ are parameterized by the Galois cohomology group $H^1(F, G)$. By \cite{rapoport-richartz} Theorem 1.15 we have a bijection
\begin{equation}\label{sequence-basic}
H^1(F, G) =\pi_1(G)_{\Gamma_F}. 
\end{equation} 
Via the map $X^*(Z(\widehat{G}))_\Gamma \cong \pi_1(G)_\Gamma \to \pi_1(G)_{\Gamma_F}$, we may associate to $H$ some length zero element $z \in \tW$. Since  by Steinberg's theorem $H\otimes_FL$ is quasi-split, we can identify $H(L)=G(L)$, and tracing through the above identifications shows that the Frobenius action induced by $H$ on $H(L)=G(L)$ is $\sigma_H = \mathop{\rm Int}(\gamma) \circ \sigma_G$; here $\gamma\in N_S(L)\subset G(L)$ is a lift of $z$ and $\mathop{\rm Int}(\gamma)$ denotes conjugation by $\gamma$.
In fact, Steinberg's theorem also applies over the maximal unramified extension $F^{nr}$ of $F$, so we can identify $H(F^{\rm nr}) = G(F^{\rm nr})$. Since conjugation by $\gamma$ preserves $S(F^{nr})$ and $T(F^{nr})$, we see that $S$ and $T$ descend to tori $S_H$, $T_H\subset H$ (over $F$). The Iwahori $I\subset G(L)$ for $G$ is also an Iwahori subgroup for $H$.

We can naturally identify $\tW$ with the Iwahori-Weyl group of $H$. This identification preserves the Coxeter structure (affine simple reflections, length, Bruhat order). Of course, the actions of $\sigma_G$ and $\sigma_H$ on $\tW$ will usually be different. Also note that while $\sigma_H$ acts on $W$, the splitting of the sequence~\eqref{tW-first-sequence} is not necessarily preserved by $\sigma_H$: Typically the set of finite simple reflections (for $G$) inside $\tW$ is not stable under  $\sigma_H$. This just reflects the fact that for non-quasi-split $H$, there is no Borel subgroup over $F$.

\subsection{$\s$-conjugacy classes}
\label{sec:sigmaconj}

We keep the notation of Section \ref{quasi-split} and draw some conclusions from results of Kottwitz \cite{kottwitz-isoI}, \cite{kottwitz-isoII} and of Rapoport and Richartz \cite{rapoport-richartz} about the classification of $\sigma$-conjugacy classes.

Denote by $B(H)$ and $B(G)$ the sets of $\sigma$-conjugacy classes in $H(L)$ (with respect to $\sigma_H$) and $G(L)$ (with respect to $\sigma_G$), respectively. The map $[b] \mapsto [b\gamma]$ is a bijection $B(H) \isomarrow B(G)$. This is the map considered by Kottwitz in \cite{kottwitz-isoII}, 4.18. We obtain the following commutative diagram
\[
\xymatrix{
\tW \ar[rr]^{x\mapsto x\gamma} \ar[d] & & \tW \ar[d] \\
B(H)\ar[rr]^{[b]\mapsto [b\gamma]} & & B(G),
}
\]
where the vertical arrows arise from the natural maps $N_S(L)\rightarrow B(H)$ and $N_S(L)\rightarrow B(G)$, respectively. Note that the map in the top row clearly preserves the length. In particular the set of length zero elements is preserved, and so the map in the bottom row maps basic elements for $H$ to basic elements for $G$.

\begin{prop}\label{s-conj}
Let $H/F$ be a connected semisimple algebraic group of adjoint type, and denote by $\tW$ its Iwahori-Weyl group (over $L$). Then the natural map $\tW\rightarrow B(H)$ is surjective.
\end{prop}

\begin{rem}
For a different proof, see~\cite{He-GeometryOfADLV}, Theorem 3.5.
\end{rem}

\begin{proof}
By the above discussion it is enough to show the proposition for the quasi-split inner form $G$ of $H$. For a quasi-split semisimple group $G$ we can prove the statement along the same lines as Corollary 7.2.2 in \cite{GHKR2}, as follows:

Recall that a $\sigma$-conjugacy class is called basic, if its Newton vector is central, i.e., is contained in the image of the map $(X_*(Z(G))_{\mathbb Q}/     W)^{\Gamma_F}\rightarrow (X_*(T)_{\mathbb Q}/W)^{\Gamma_F}$. Let us first show that the set of length zero elements in $\tW$ maps surjectively to the subset of basic $\sigma$-conjugacy classes in $B(G)$.


This follows from the fact that the identification~\eqref{sequence-basic} is compatible with the natural maps to $B(G)$ and the fact that $H^1(F,G)$ maps bijectively to the set of basic $\sigma$-conjugacy classes. See e.g.~\cite{rapoport-richartz}, Theorem 1.15.

Now let $[b]\in B(G)$ be an arbitrary $\sigma$-conjugacy class. By the description in \cite{kottwitz-isoI} Proposition 6.2, we may assume that $b$ is contained in a Levi subgroup $M\subseteq G$ attached to a standard parabolic subgroup, and is basic for $M$. By the above, it follows that the $\sigma$-conjugacy class of $b$, viewed as an element of $B(M)$ is in the image of the Iwahori-Weyl group $\tW_M$ of $M$. Because the obvious diagram
\[
\xymatrix{
\tW_M \ar[r]\ar[d] & \tW\ar[d] \\
B(M)\ar[r] & B(G)
}
\]
is commutative, this proves the proposition.
\end{proof}

\subsection{Affine Deligne-Lusztig varieties}
We can also identify affine Deligne-Lusztig varieties for $G$ and $H$. Recall that the Iwahori $I\subset G(L)$ for $G$ is at the same time an Iwahori subgroup for $H$, so that we can identify the affine flag varieties for $G$ and for $H$. Furthermore $I$ is normalized by $\gamma$, because the length zero elements stabilize the base alcove. For any $x \in \tW$ and $b \in G(L)=H(L)$, the condition $g^{-1} b\sigma_H(g) \in IxI$
precisely amounts to $g^{-1}b\gamma\sigma_G(g)\gamma^{-1} \in IxI = Ix\gamma I\gamma^{-1}$. Thus

\begin{prop}
Let $G$, $H$ and $\gamma$ be as above. Let $x\in \tW$, and let $b\in G(L)=H(L)$. Then
\[
X^H_x(b) = X^G_{x\gamma}(b\gamma).
\]
\end{prop}

\section{$P$-alcoves and emptiness of ADLV}
\label{SEC:emptiness}

In the rest of this paper, we let $G$ be a quasi-split connected semisimple group over $F$ that \emph{splits over a tamely ramified extension of $L$}. We simply write $\sigma$ for the Frobenius map $\sigma_G$ on $G(L)$ and write $\d$ for the induced automorphisms on $W$ and $\tW$.

\subsection{The root system}

Consider the real vector space $V=X_*(T)_\G \otimes \RR$. Let $\Phi$ be the set of (relative) roots of $G$ over $L$ with respect to $S$ and $\Phi_a$ the set of affine roots. The roots in $\Phi$ determine hyperplanes in $V$ and the relative Weyl group $W$ can be identified with the group generated by the reflections through these hyperplanes. 

Note that the root system $\Phi$ is not necessarily reduced. By \cite{Tits}, Section 1.7, there exists a unique reduced root system $\Sigma$ such that the affine roots $\Phi_a$ consists of functions on $V$ of the form $y \mapsto \a(y)+k$ for $\a \in \Sigma$ and $k \in \ZZ$. Moreover, $W=W(\Sigma)$ and $W_a=Q^\vee(\Sigma) \rtimes W$. Here $W(\Sigma)$ is the Weyl group of the root system $\Sigma$ and $Q^\vee(\Sigma)$ is the coroot lattice for $\Sigma$. 

Note that any root of $\Sigma$ is proportional to a root in $\Phi$. However, the root system $\Sigma$ is not necessarily proportional to $\Phi$, even if $\Phi$ is reduced. See \cite{Tits}, Section 1.7. 

Of course the length function and Bruhat order on $W_a$ produced in these two ways are the same, since in both cases they are given by the affine root hyperplanes in $V$, which are the same in both cases. The identification with the affine Weyl group of a reduced root system allows us to use the corresponding notions and results from the theory of root systems.

\subsection{Parabolic subgroup}
For $a \in \Phi$, we denote by $\mathbf U_a \subset G_L$ the corresponding root subgroup and for $\a \in \Phi_a$, we denote by $U_\a \subset L(G_L)$ the corresponding root subgroup scheme over $\kk$. They are described explicitly in \cite{pappas-rapoport:twisted}, Sections 9.a, 9.b. For the definition the assumption of loc.cit.~that $G$ be simply connected does not play a role. In particular, $U_\a$ is one-dimensional for all $\a \in \Phi_a$.

Our choice of fundamental alcove determines a basis $\mathbb S$ of $\Phi$. We choose the same normalization as in \cite{GHKR2}, which means that the fundamental alcove lies in the anti-dominant Weyl chamber.  We identify $\mathbb S$ with the set of simple reflections in $W$ and hence can also view $\mathbb S$ as a basis of the reduced root system $\Sigma$. Let $\Phi^+$ (resp. $\Phi^-$) be the set of positive (resp.~negative) roots of $\Phi$. For $J \subset \SS$, let $\Phi_J$ be the set of roots spanned by $J$ and let $\Phi^\pm_J=\Phi_J \cap \Phi^\pm$. Then $J$ is a basis of the subsystem $\Phi_J$. Let $W_J \subset W$ be the corresponding standard parabolic subgroup and $Q^\vee_J$ be the corresponding coroot lattice. 

We denote by $M_J$ the subgroup of $G$ generated by $T$ and $\mathbf U_a$ for $a \in \Phi_J$ and by $N_J$ the subgroups generated by $\mathbf U_a$ for $a \in \Phi^+-\Phi^+_J$. Let $P_J$ be the subgroup generated by $M_J$ and $N_J$. By \cite{Springer:LinearAlgGroups}, Section 15.4, $P_J=M_J N_J$ is a parabolic subgroup of $G$. If moreover $\d(J)=J$, then $P_J, M_J$ and $N_J$ are defined over $F$. The Iwahori-Weyl group of $M_J$ is $\tW_J=X_*(T)_\G \rtimes W_J$. We simply write $\k_J$ instead of $\k_{M_J}$.

\subsection{$(J, w, \d)$-alcoves}
\label{sec:Jwd}

As in \cite{GHKR2}, we use the notation ${}^x g:= xgx^{-1}$ and ${}^\sigma g :=\sigma(g)$ for $g\in G(L)$, and similarly for subsets of $G(L)$.

Let $J \subset \SS$ with $\d(J)=J$ and $w \in W$. Let $x\in\tW$.  We say $x{\bf a}$  is a \emph{$(J, w, \d)$-alcove}, if
\begin{enumerate}
\item[(1)] $w \i x \d(w) \in \tW_J$, and
\item[(2)] For any $a \in w(\Phi^+-\Phi^+_J)$, $\mathbf U_a \cap {}^x I \subseteq \mathbf U_a \cap I$, or equivalently, $\mathbf U_{-a} \cap {}^x I \supseteq \mathbf U_{-a} \cap I$.
\end{enumerate}
We say $x{\bf a}$ is a \emph{strict} $(J, w, \d)$-alcove if instead of (2) we have
\begin{enumerate}
\item[($3$)] For any $a \in w(\Phi^+-\Phi^+_J)$, $\mathbf U_a \cap {}^x I \subsetneqq \mathbf U_a \cap I$, or equivalently, $\mathbf U_{-a} \cap {}^x I \supsetneqq \mathbf U_{-a} \cap I$.
\end{enumerate}

In the split case, $x{\bf a}$ is a $(J, w, \d)$-alcove if and only if it is a ${}^w P_J$-alcove in the sense of \cite{GHKR2}. 

Condition (1) implies that ${}^{x \s} ({}^w M_J)={}^w M_J$. If we pass to the (non-connected) group $G \rtimes \langle\sigma\rangle$, then we can reformulate condition (1) above as $x\delta \in {}^w(\tW_J\rtimes\langle\delta\rangle)$.


Now we state the main result of this section, which generalizes \cite{GHKR2}, Theorem 2.1.2.

\begin{theorem} \label{mainthm}
Suppose $J \subset \SS$ with $\d(J)=J$ and $w \in W$, and $x\bf a$ is a $(J, w, \d)$-alcove. Set $I_M={}^w M_J \cap I$.
Then the map
$$
\phi: I \times^{I_M} I_M x \s(I_M) \rightarrow IxI
$$
induced by $(i,m) \mapsto im \sigma(i)^{-1}$, is surjective.  If $x{\bf a}$ is a strict $(J, w, \d)$-alcove, then $\phi$ is injective.  In general, $\phi$ is not injective, but if $[i,m]$ and $[i',m']$ belong to the same fiber of $\phi$, the elements $m$ and $m'$ are $\sigma$-conjugate by an element of $I_M$.
\end{theorem}

Similar to \cite{GHKR2}, Lemma 4.1.1, the theorem is equivalent to the following statement: the map
$$
\phi: \,\, ({}^{\d \i(x) \i} I \cap I) \times^{\, ^{\d \i(x) \i}I_M \cap I_M} I_M x \rightarrow I x
$$
given by $(i,m) \mapsto im\, \sigma (i)^{-1}$ is surjective.   It is bijective if $x{\bf a}$ is a strict $(J, w, \d)$-alcove.  In general, if $[i,xj]$ and $[i',xj']$ belong to the same fiber of $\phi$, then $xj$ and $xj'$ are $\sigma$-conjugate by an element of $^xI_M \cap I_M$.

The proof of the portion relating to the fiber of $\phi$ is just the same as in \cite{GHKR2}, Section 4. For the proof of surjectivity, we follow the strategy of \cite{GHKR2}, Section 6.

\subsection{} For $n \in \NN$, let $T(L)_n$ be the corresponding congruence subgroup of $T(L) \cap I$ (see \cite{prasad-raghunathan} 2.6). For any $r \geq 0$, let $I_r \subset I$ be the subgroup generated by $T(L)_n$ for $n \geq r$ and $U_{a+m}$ for $a \in \Phi$ and $m \geq r$ such that $a+m$ is a positive affine root. Let $I_{r^+}=\cup_{s>r} I_s$. Then $I_r$ and $I_{r^+}$ are normal subgroups of $I$ for all $r \geq 0$.

Recall that $x {\bf a}$ is a $(J, w, \d)$-alcove. Let $M={}^w M_J$. Let $N \subset G$ be the subgroup generated by $\mathbf U_a$ for $a \in w(\Phi^+-\Phi^+_J)$ and $\overline N \subset G$ be the subgroup generated by $\mathbf U_{-a}$ for $a \in w(\Phi^+-\Phi^+_J)$.

For $r \geq 0$, let $N_r=N(L) \cap I_r$ and $N_{r^+}=N(L) \cap I_{r^+}$. They are normal subgroups of $N(L) \cap I$. Similarly, let $\overline{N}_r=\overline{N}(L) \cap I_r$ and $\overline{N}_{r^+}=\overline{N}(L) \cap I_{r^+}$.

Since $x {\bf a}$ is a $(J, w, \d)$-alcove, we have ${}^{x \s} N_r \subseteq N_r$ and ${}^{x \s} \overline N_r \supseteq \overline N_r$.

\begin{lemma} \label{sided_approx_lemma}  Fix an element $m \in I_M$ and $r \geq 0$.
\begin{enumerate}
\item[(i)] Given $i_- \in \s(\overline{N}_r)$, there exists $b_- \in \overline{N}_r$ such that
$^{(m x)^{-1}}b_- i_- \, ^{\sigma}b_-^{-1} \in \s(\overline{N}_{r^+})$.
\item[(ii)] Given $i_+ \in N_r$, there exists $b_+ \in N_r$ such that
$b_+ i_+ \, ^{mx\sigma}b_+^{-1} \in N_{r^+}$.
\end{enumerate}
\end{lemma}

\begin{proof}

To the Borel subgroup ${}^w P_\emptyset$ of $G$, we associate a finite separating
filtration by normal subgroups
$$
N_L = N[1] \supset N[2] \supset \cdots
$$
as in \cite{GHKR2}, proof of Lemma 6.1.1.

This filtration has the following properties:
\begin{enumerate}
\item
For each $i$, $N[i]\subset N_L$ is normal, and stable under conjugation with elements of $M$.
\item
For each $i$, ${}^{x\sigma}N[i] \subseteq N[i]$.
\item
For each $i$, the quotient $N\langle i \rangle := N[i]/N[i+1]$ is abelian.
\end{enumerate}
We define $N_r[i] := N_r \cap N[i]$, and $N_r\langle i \rangle := N_r[i]/N_r[i+1]$, and define $N_{r^+}\langle i\rangle$ analogously.  Then $N_r\langle i\rangle/N_{r^+}\langle i\rangle$ is a vector group over $\kk$. We define the groups
$\overline{N}[i]$, $\,\overline{N}\langle i \rangle$, $\,\overline{N}_r[i]$, $\overline{N}_{r}\< i\>$ and $\overline{N}_{r^+}\langle i \rangle$ in an analogous manner. It is easy to see from the definition that ${}^{(m x) \i} \overline{N}_r[i] \subset \s(\overline{N}_r[i])$ and $^{mx\sigma} N_r[i] \subset N_r[i]$.

By \cite{GHKR2}~Lemma 5.1.1, the map
$b_- \mapsto \, ^{(m x)^{-1}}b_- \, ^{\sigma}b_-^{-1}$ is surjective from the vector
group $\overline{N}_{r}\langle i \rangle/\overline{N}_{r^+}\langle i
\rangle$ to $\s(\overline{N}_{r}\langle i \rangle/\overline{N}_{r^+}\langle i
\rangle)$ and the map $b_+ \mapsto \, b_+ \, ^{mx\sigma}b_+^{-1}$ is surjective on each vector group $N_r\langle i \rangle/{N}_{r^+}\langle i
\rangle$.  Applying it repeatedly on these quotients in a suitable order, we may find
$b_- \in \overline{N}_r$ such that
$$
^{(xm)^{-1}}b_- i_- \, ^{\sigma}b_-^{-1} \in \overline{N}_{r^+},
$$
and $b_+ \in N_r$ such that
$b_+ i_+ \, ^{mx\sigma}b_+^{-1} \in N_{r^+}$.
\end{proof}

\begin{cor}
Let $m \in I_M$ and $r \geq 0$. Given $i_- \in \overline{N}_r$, there exists $b_- \in \s \i({}^{(m x) \i} \overline{N}_r)$ such that
$b_- i_- \, ^{m x \sigma}b_-^{-1} \in \overline{N}_{r^+}$.
\end{cor}

\begin{proof}
By Lemma \ref{sided_approx_lemma}, there exists $b \in \overline{N}_r$ such that $^{(m x)^{-1}}b \s(i_-) \, ^{\sigma}b^{-1} \in \s(\overline{N}_{r^+})$. Set $b_-=\s \i(^{(m x)^{-1}}b)$. Then $b_- i_- \,^{mx\sigma}b_-^{-1} \in \overline{N}_{r^+}$.
\end{proof}

As explained in \cite{GHKR2}, Section 6, a generic Moy-Prasad filtration gives a filtration $I=\cup_{r \ge 0} I[r]$ with $I[r] \supset I[s]$ for $r<s$ satisfying the following conditions:
\begin{enumerate}
\item Each $I[r]$ is normal in $I$.
\item Each $I[r]$ is a semidirect product $I\langle r \rangle I[r^+]$,
where $I\langle r \rangle$ is either an affine root subgroup (hence
one-dimensional over our ground field $k$) or else contained in $T(\mathfrak o)$.
\end{enumerate}

Let $ y \in I x$. By the same argument as in \cite{GHKR2}, Section 6, for any $i \ge 0$, there exists $h_i \in {}^{\d \i(x) \i} I \cap I$ (suitably small when $i$ is large) such that $$h_i h_{i-1} \cdots h_0 y \s(h_i h_{i-1} \cdots h_0) \i \in I[i^+] I_M x.$$ Let  $g= \cdots h^{(2)}h^{(1)}h^{(0)}$ be the convergent product. Then $g y \s(g) \i  \in x I_M$. This proves the surjectivity.
\

By the same argument as in \cite{GHKR2}, Section 6, we also have the following result.

\begin{prop}
Suppose $J \subset \SS$ with $\d(J)=J$ and $w \in W$, and $x\bf a$ is a $(J, w, \d)$-alcove. Set $I_M={}^w M_J \cap I$. If moreover, $^{x \s} I_M=I_M$, then we may $\s$-conjugate any element of $I x$ to $x$, using an element of ${}^{\d \i(x) \i} I \cap I$.
\end{prop}

\subsection{Some properties on Newton vectors}\label{newton}

Let $n$ be the order of $W \rtimes \<\d\>$ (we consider $\d$ as an element of the automorphism group of $W$). For any $\l \in X_*(T)_\G$, set $\nu_\l=\sum_{i=0}^{n-1} \d^i(\l)/n \in V$. For $x \in \tW$, $x \d(x) \cdots \d^{n-1}(x)=\e^{\mu}$ for some $\mu \in X_*(T)^\d_\G$. We set $\nu_x=\mu/n \in X_*(T)^\d_\G \otimes \QQ$. It is easy to see that $\nu_\l=\nu_{\e^\l}$ for any $\l \in X_*(T)_\G$. 

For any $b \in G(L)$, let $\bar \nu_b$ be the  (dominant) Newton vector of $b$. If $x \in \tW$ and $\dot{x}\in N_S(L)\subseteq G(L)$ is a representative of $x$, then $\bar \nu_{\dot{x}}$ is the unique dominant element in the $W$-orbit of $\nu_x$. 

The following properties are easy to verify and we omit the details. 

(1) Let $J \subset \SS$ with $\d(J)=J$ and $x=\e^\l w \in \tW_J$. Then $\nu_x-\nu_\l \in Q^\vee_J \otimes_\ZZ \QQ$. 


(2) Assume that the Dynkin diagram of $G$ is $\d$-connected. Let $J \subsetneqq \SS$ with $\d(J)=J$. If $\l, \l' \in V$ such that $\<\l, \a\> \ge 0$ for all $\a \in J$, $\l'$ is central and $\l-\l' \in Q^\vee_J \otimes_\ZZ \QQ$, then $\l=\l'$. 

\

The following proposition says that $\s$-conjugacy classes never fuse.


\begin{prop}\label{fuse}
Let $[b]$ be a $\sigma$-conjugacy class in $G(L)$ and $J \subset \SS$ with $\d(J)=J$. Then $[b] \cap M_J(L)$ contains at most one $\s$-conjugacy class of $M_J(L)$.
\end{prop}

\begin{proof}
By Proposition \ref{s-conj}, any $\s$-conjugacy class of $M_J(L)$ is represented by some element in $\tW_J$. Let $x=\e^\l w, x'=\e^{\l'} w' \in \tW_J$ such that $x$ and $x'$ are in the same $\s$-conjugacy class of $G(L)$. By Kottwitz \cite{kottwitz-isoI} and \cite{kottwitz-isoII},  $\nu_x=\nu_{x'}$ and $\k_G(x)=\k_G(x')$. We have $\l'=\l+\th-\d(\th)+r_J+r'_J$ for some coweight $\th$, $r_J \in Q^\vee_J$ and $r'_J \in Q^\vee_{\SS-J}$. 

By Section \ref{newton} (1), $\nu_{\l'}-\nu_\l \in Q^\vee_J \otimes_\ZZ \QQ$. Hence $$\nu_{r'_J} \in Q^\vee_J \otimes_\ZZ \QQ \cap Q^\vee_{\SS-J} \otimes_\ZZ \QQ=\{0\}.$$ 
In other words, $\sum_{i=0}^{n-1}\d^i(r'_J)=0$, where $n$ is the order of $W \rtimes \<\d\>$. 
Since $\d$ permutes simple coroots of $\SS-J$, we can assume, without loss of generality, that $r'_J=\sum_{j=0}^{s-1} b_j\d^j(\a^\vee)$, where $b_j \in \ZZ$, $\a^\vee$ is a simple coroot of $\SS-J$ and $s$ is the smallest positive integer with $\d^s(\a^\vee)=\a^\vee$. The equality $\sum_{i=0}^{n-1}\d^i(r'_J)=0$ is equivalent to $\sum_{j=0}^{s-1} b_j=0$. Let $c_j=\sum_{k=0}^j b_k$ and $v=\sum_{j=0}^{s-1}c_j \d^j(\a^\vee)$. Then $r'_J=v-\d(v)$. Hence $\l'-\l=\th'-\d(\th')+r_J$ for some coweight $\th'$. 

Therefore $\k_J(x)=\k_J(x')$. By \cite{kottwitz-isoII} 4.13, $x$ and $x'$ are in the same $\s$-conjugacy class of $M_J(L)$. 
\end{proof}

\subsection{Applications to affine Deligne-Lusztig varieties}

We consider the following maps from the Iwahori-Weyl group $\tW$ to the finite Weyl group $W$:
\begin{align*}
& \eta_1\colon \tW = X_*(T)_\Gamma \rtimes W \rightarrow W, \text{ the projection}\\
& \eta_2(x) \text{ is the unique element $v$ such that $v^{-1}x\in {^\SS}\tW$}\\
& \eta_\d(x) = \delta^{-1}(\eta_2(x)^{-1}\eta_1(x))\eta_2(x).
\end{align*}

Here ${^\SS}\tW$ is the set of $x \in \tW$ such that $x\mathbf a$ lies in the dominant chamber.
So if $x=v\e^\mu w$ with $\e^\mu w\mathbf a$ contained in the dominant chamber, $v, w\in W$, then $\eta_1(x)=vw$, $\eta_2(x)=v$, and $\eta_\d(x) = \delta^{-1}(w)v$.

Now we discuss some consequences on affine Deligne-Lusztig varieties. For analogues in the split case, see \cite{GHKR2} Section 9.

\begin{cor}\label{ConsequBasic}
Let $[b]$ be a basic $\sigma$-conjugacy class in $G(L)$. Suppose $J \subset \SS$ with $\d(J)=J$ and $w \in W$, and $x\bf a$ is a $(J, w, \d)$-alcove.  Then $X_x(b)=\emptyset$,
unless $\kappa_J(w \i x \d(w)) \in \kappa_J([b] \cap M_J(L))$.
\end{cor}

\begin{rem}
By Proposition \ref{fuse}, $[b] \cap M_J(L)$ is empty or a single $\s$-conjugacy class of $M_J(L)$ and hence $\kappa_J([b] \cap M_J(L))$ consists of at most one element.
\end{rem}

\begin{lemma}\label{3.6.3}
Let $J \subset \SS$ with $\d(J)=J$. Let $x\in\tW$, and write $w = \eta_2(x)\in W$. If $\eta_\d(x)\in \tW_J$, then $x {\bf a}$ is a $(J, w, \d)$-alcove.
\end{lemma}

\begin{proof} Let $U$ be the subgroup of $G$ generated by $\mathbf U_\a$ for $\a \in \Phi^+$. Then for any $\b \in w(\Phi^+-\Phi^+_J)$,
\[
\mathbf U_\b \cap {}^x I \subseteq {}^w U \cap {}^{x} I \subseteq {}^w (U \cap I) \subseteq I.
\]
The second inclusion follows from the assumption that $w \i x {\bf a}$ lies in the dominant chamber.
\end{proof}

\begin{prop} \label{reuman1}
Assume that the Dynkin diagram of $G$ is $\d$-connected. Let $b$ be basic.
Let $x \in \tW$, and write $x = \epsilon^\lambda u$, $u\in W$.
Assume that $\overline{\nu}_{\eta_2(x) \i \lambda} \ne \overline{\nu}_b$ and that
$\eta_\delta(x)\in \bigcup_{J \subsetneq \SS, \d(J)=J}W_J$. Then $X_x(b) = \emptyset$.
\end{prop}

\begin{proof}
Write $w=\eta_2(x) \in W$. By Lemma \ref{3.6.3} and our hypothesis, $x{\bf a}$ is a $(J, w, \d)$-alcove for some $\d$-stable proper subset $J\subsetneq \SS$.
 The only thing we need to check in order to apply
Corollary \ref{ConsequBasic} is that $\kappa_{J}(w \i x \d(w))\not \in \kappa_{J}([b]\cap M_J(L))$. Here we denote by $[b]\subset G(L)$ the $\sigma$-conjugacy class of $b$. Otherwise, there exists $b_J\in M_J(L)$ which is $\sigma$-conjugate to $b$, and such that $\kappa_J(w\i x\d(w)) = \kappa_J(b_J)$. We may and will assume that $b_J\in \tW_J$. If we write $w\i x \d(w) = \e^{\lambda'} u'$, $b_J = \e^\mu v$, $u', v\in W_J$, then $\lambda'=w \i \l$ and for a suitable coweight $\theta$,
\[
\lambda' -\mu +\theta-\d(\theta) \in Q^\vee_J.
\]

Thus $\nu_{\l'}-\nu_\mu \in Q^\vee_J \otimes_\ZZ \QQ$. By Section \ref{newton} (1), $\nu_{\l'}-\nu_{b_J} \in Q^\vee_J \otimes_\ZZ \QQ$. Note that $\l'$ is dominant and $\nu_{b_J}=\bar \nu_b$ is central. By Section \ref{newton} (2), $\nu_{\l'}=\bar \nu_b$, which we have ruled out by assumption.
\end{proof}


Following \cite{GHKR2}, for any $a \in \Sigma$ and alcove ${\mathbf b}$, let $k(a, \mathbf b)$ be the unique integer $k$ such that $\mathbf b$ lies in the region between the hyperplanes $H_{a, k}$ and $H_{a, k-1}$.

\begin{prop}
\label{reuman3}
Let $x = \e^\l u$ lie in the shrunken Weyl chambers,  i.e.~$k(a, x\mathbf a)\ne k(a, \mathbf a)$ for all $a\in\Sigma$. Assume that $\eta_\d(x)\in\bigcup_{J\subsetneq \SS, \d(J)=J} W_J$. Then $\overline{\nu}_{\eta_2(x) \i \l}$ is not central.
 
If moreover, the Dynkin diagram of $G$ is $\d$-connected, then $X_x(b)=\emptyset$ for any basic element $b \in G(L)$.
\end{prop}

\begin{proof}
Let $n$ be the order of $W \rtimes \<\d\>$. Let $\l'=\eta_2(x) \i \l$. Suppose that $\overline{\nu}_{\l'}$ is central. Then $\l'+\d(\l')+\cdots+\d^{n-1}(\l')$ is central and $$\<\l'+\d(\l')+\cdots+\d^{n-1}(\l'), \b\>=n\<\l', \b\>=0,$$ where $\b$ is the unique maximal root. As $\l'$ is dominant, $\l'$ is central. Hence $x=\eta_2(x) \e^{\l'}=u \e^{\l'}$. Thus $x \mathbf a=u \mathbf a$. This alcove belongs to the shrunken Weyl chambers only if $u=w_0$. This contradicts our assumption that $\eta_\d(x)\in\bigcup_{J\subsetneq \SS, \d(J)=J} W_J$.

The ``moreover'' part follows from Proposition \ref{reuman1}. 
\end{proof}

\section{Reduction method and nonemptiness of ADLV}
\label{SEC:nonemptiness}


\subsection{Condition (2) of $(J, w, \d)$-alcoves} 
For any $a \in \Sigma$ and alcoves $\mathbf b_1$ and $\mathbf b_2$, we say that ${\mathbf b}_1 \ge_a {\mathbf b}_2$ if $k(a, {\mathbf b}_1) \ge k(a, {\mathbf b}_2)$. 

Condition (2) is equivalent to saying that for any $a \in w(\Phi^+-\Phi^+_J)$ and an affine root $\a=a+m$ (with $m \in \mathbb Q$), if $x{\mathbf a}$ is in the half-apartment $\a \i({[-\infty, 0]})$, then so is ${\mathbf a}$. We may then reformulate this definition as follows. 

(2') For any $a \in w(\Sigma^+-\Sigma^+_J)$, $x \mathbf a \ge_a \mathbf a$. 

In particular, this condition is just a condition on the relative position between certain alcoves and walls. Thus it only depends on the affine Weyl group and does not depends on the set of affine roots. 

\begin{prop} \label{reuman2}
Let $x\in\tW$ lie in the shrunken Weyl chambers. If $x$ is a $(J, w, \d)$-alcove for $J\subseteq \SS$ with $\d(J)=J$ and $w\in W$, then $\eta_\d(x)\in W_J$.
\end{prop}

\begin{proof}
By the definition of the shrunken Weyl chambers and of $(J, w, \d)$-alcoves, for any $a \in w(\Sigma^+-\Sigma^+_J)$, 
\[
k(\eta_2(x) \i a, \eta_2(x) \i x \mathbf a)=k(a, x \mathbf a)>k(a, \mathbf a) \ge 0. 
\]

Since $\eta_2(x)\i x\mathbf a$ lies in the dominant chamber, $\eta_2(x) \i a \in \Sigma^+$ for all $a \in w(\Sigma^+-\Sigma^+_J)$. Therefore $\eta_2(x) \i w \in W_J$. By the definition of $(J, w, \d)$-alcoves, $w^{-1}\eta_1(x)\d(w)\in W_J$. Thus $\eta_\d(x) \in W_J$. 
\end{proof}


\subsection{Reduction method} In this section, we will recall the reduction method in \cite{He-GeometryOfADLV} and prove that $P$-alcoves are ``compatible'' with the reduction. As a consequence, we prove that an affine Deligne-Lusztig variety $X_w(b)$ for basic $b$ is nonempty exactly when the $P$-alcoves predict it to be. See Theorem~\ref{thm-nonemptiness} for the precise formulation; compare also with Corollary~\ref{ConsequBasic}.

We first recall a ``reduction method'' \`a la Deligne and Lusztig \cite[proof of Theorem 1.6]{DL}, compare also~\cite{GoertzHe1}.

\begin{prop}\label{dl-red}
Let $b \in G(L)$, $x \in \tW$ and $s \in \tSS$.

(1) If $\ell(s x \d(s))=\ell(x)$, then $X_x(b) \neq \emptyset$ if and only if $X_{s x \d(s)}(b) \neq \emptyset$.

(2)  If $\ell(s x \d(s))=\ell(x)-2$, then $X_x(b) \neq \emptyset$ if and only if $X_{s x \d(s)}(b) \neq \emptyset$ or $X_{s x}(b) \neq \emptyset$.
\end{prop}

\subsection{Minimal length elements}
Let $x, x' \in \tW$ and $s \in \tSS$. We write $x \overset s \to_\d x'$ if $x'=sx \d(s)$ and $\ell(x) \geq \ell(x')$ and write $x \overset s \rightharpoonup x'$ if either $x \overset s \to_\d x'$ or $x'=s x$ and $\ell(x)>\ell(x')$.

We write $x \to_\d x'$ if there exists a sequence $x_0, x_1, \cdots, x_r$ in $\tW$ and a sequence $s_1, s_2, \cdots, s_r$ in $\tSS$ such that $x=x_0 \overset {s_1} \to_\d x_1 \overset {s_2} \to_\d \cdots \overset {s_r} \to_\d x_r=x'$.  Similarly, we may define $x \rightharpoonup x'$.

We define the {\it $\d$-twisted conjugation action} of $\tW$ on itself by $w \cdot_\d w'=w w' \d(w) \i$. Any orbit is called a {\it $\d$-twisted conjugacy class} of $\tW$. For any $\d$-twisted conjugacy class $\co$ of $\tW$, we denote by $\co_{\min}$ the set of minimal length elements in $\co$.

One of the main results in \cite{HN} is

\begin{theorem}\label{min-co}
Let $\co$ be a $\d$-twisted conjugacy class of $\tW$. Then for any $x  \in \co$, there exists $x' \in \co_{\min}$ such that $x \to_\d x'$.
\end{theorem}

Note that \cite{HN} does also include the twisted case; there the action of $\d$ is incorporated by replacing $\tW$ by a semi-direct product of the form $\tW \rtimes \<\d\>$.

The following result is a consequence of the ``degree=dimension'' theorem in \cite{He-GeometryOfADLV}. We include here a proof for completeness.

\begin{theorem}\label{dxd}
Let $x \in \tW$ and $D_{x, \d}$ be the set of elements $y \in \tW$ such that $y$ is of minimal length in its $\d$-twisted conjugacy class and $x \rightharpoonup y$. Then for any $b \in G(L)$, $X_x(b) \neq \emptyset$ if and only if $X_y(b) \neq \emptyset$ for some $y \in D_{x, \d}$.
\end{theorem}

\begin{proof}
If $y \in D_{x, \d}$ and $X_y(b) \neq \emptyset$, then by Proposition \ref{dl-red} and the definition of $D_{x, \d}$, $X_x(b) \neq \emptyset$. Now we assume that $X_x(b) \neq \emptyset$. We proceed by induction on the length of $x$.

If $x$ is a minimal length element in its $\d$-twisted conjugacy class $\co$, then $x \in D_{x, \d}$. The statement is obvious.

Suppose that $x$ is not a minimal length element in its $\d$-twisted conjugacy class. By Theorem \ref{min-co}
there exists $x' \in \tW$ and $s \in \tSS$ such that $x \to_\d x'$, $\ell(x)=\ell(x')$ and $\ell(s x' \d(s))=\ell(x')-2$. By Proposition \ref{dl-red}, $X_{x'}(b) \neq \emptyset$ and $X_{s x' \d(s)}(b) \neq \emptyset$ or $X_{s x'}(b) \neq \emptyset$. Since $\ell(s x' \d(s)), \ell(s x')<\ell(x)$, by induction hypothesis, there exists $y \in D_{s x' \d(s), \d} \cup D_{s x', \d}$ such that $X_y(b)\ne\emptyset$. By definition, $D_{s x' \d(s), \d} \cup D_{s x', \d} \subset D_{x, \d}$. So $y \in D_{x, \d}$. The statement holds for $x$.
\end{proof}

\subsection{Property (NLO)}
\label{sec:PropertyStar}

We now fix a basic element $b\in \tW$.

\begin{defn}
We say that $y\in \tW$ has property (NLO) (with respect to $b$), if for every pair $(J, w)$ with $J \subset \SS$, $\d(J)=J$ and $w \in W$, such that $y$ is a $(J, w, \d)$-alcove, there exists $b_{J} \in w \tW_{J} \d(w)\i$ such that
\begin{enumerate}
\item
$\kappa_G(b) = \kappa_G(b_{J})$,
\item
$\nu_{b_{J}} = \nu_b$,
\item
$\kappa_{J}(w\i b_{J} \d(w)) = \kappa_{J}(w\i y\d(w))$.
\end{enumerate}
\end{defn}

Here (NLO) stands for \emph{no Levi obstruction}: Heuristically, affine Deligne-Lusztig varieties should be non-empty, unless there is an evident obstruction. For instance, if $\kappa_G(b) \ne \kappa_G(x)$, then $X_x(b)=\emptyset$, as is easily checked. Moreover, as the previous results show, an obstruction of a similar kind can originate from other Levi subgroups of $G$. This kind of obstruction is formalized in the above definition, and we will see that it is in fact the only obstruction to non-emptiness.

By Theorem \ref{dxd}, to prove the nonemptiness, one only needs to examine the claim for the reduction step and for minimal length elements. 

\begin{lemma}\label{lem.2.1}
Denote by $x \mapsto \bar{x}$ the projection $\tW \rightarrow W$. Let $y \in \tW$ and $s \in \tilde \SS$. Assume that $s=s_H$ for some affine root hyperplane $H=H_{\a,k}$ with $\a \in \Sigma$ and $k \in \ZZ$. Let $\b \in \Sigma$.
\begin{enumerate}
\item If $\b \notin \{\pm \a, \pm \bar y\d(\a)\}$, then $sy\d(s) \ba \ge_{\bar s(\b)} \ba$  if and only if $y\ba \ge_\b \ba$ .
\item If $\b \neq \pm \bar y\d(\a)$, then $y\ba \ge_\b \ba$ if and only if $y\d(s) \ba \ge_\b \ba$.
\end{enumerate}
\end{lemma}

\begin{proof}
We only prove (1). (2) can be proved in the same way.

By the assumption on $\b$, there exists a point $e \in \overline{\ba} \cap s\overline{\ba} \subset H$ such that, with $e':=y\d(e) \in y \bar{\ba} \cap y \d(s)\bar{\ba} \subset y\d H$, we have $(e, \b), (e', \b) \notin \ZZ$. The statement follows from the following fact which is easily checked:

Let $\mathbf c\ne \mathbf c', \mathbf d\ne\mathbf d'$ be alcoves such that there exist $e \in \overline{\mathbf c}\cap\overline{\mathbf c'}$, $e' \in \overline{\mathbf d}\cap\overline{\mathbf d'}$, and let $\b \in \Sigma$ with $(e, \b), (e',\b) \not\in\ZZ$. Then $\mathbf c \ge_{\b} \mathbf d$ if and only if $\mathbf c' \ge_{\b} \mathbf d'$.
\end{proof}

\begin{lemma}\label{lem.2.3}
Let $y \in \tW$ and $s \in \tSS$ with $\ell(s y \d(s))=\ell(y)$. If $y \ba$ is a $(J, w, \d)$-alcove, then  $sy\d(s)\ba$ is a $(J, \bar s w, \d)$-alcove.
\end{lemma}

\begin{proof}
It suffices to show that $sy\d(s) \ba \ge_\b \ba$ for $\b \in \bar s w(\Sigma^+- \Sigma^+_J)$.

Assume that $s=s_H$ for some affine root hyperplane $H=H_{\a,k}$ with $\a \in \Sigma$ and $k \in \ZZ$. If $\b \notin \{\pm \a, \pm \bar s \bar y\d(\a)\}$, the statement follows from Lemma \ref{lem.2.1}.

Without loss of generality, we assume that $-\a \in \bar s w(\Sigma^+-\Sigma^+_J)$. Then $w \i(\a), \d(w \i (\a)) \in \Sigma^+-\Sigma^+_J$. So $w \i \bar y \d(w) \d(w \i (\a))=w \i \bar y \d(\a) \in \Sigma^+-\Sigma^+_J$ and $\bar s \bar y \d(\a) \in \bar s w(\Sigma^+-\Sigma^+_J)$.

It remains to show that $sy\d(s) \ba \ge_{-\a} \ba $ and $sy\d(s) \ge_{\bar s\bar y\d(\a)} \ba$. There are two cases.

Case 1: $H=y\d(H)$. Then $-\a=\bar s \bar y\d(\a)$. So $\ba, sy\d(s) \ba$ are on the same side of $H$ and their closures intersect with $H$. Hence $sy\d(s) \ba=_\a \ba$.

Case 2: $H \neq y\d(H)$. Without loss of generality, we assume that $\ell(y\d(s))<\ell(y)$ (arguments for the case $\ell(y\d(s))>\ell(y)$ are similar). In this case, $y\d H$ separates $y\ba$ from $y\d(s)\ba$ and $\ba$. Since $y \ba$ is a $(J, w, \d)$-alcove, $y\ba \ge_{\bar y\d(\a)} \ba$. Hence $y\ba \ge_{\bar y\d(\a)} y\d(s)\ba$ and $\ba \ge_\a s\ba$. Since $\ell(y\d(s))<\ell(sy\d(s))=\ell(y)$, $s\ba, sy\d(s)\ba$ are on the same side of $H$, therefore $sy\d(s)\ba \ge_{-\a} \ba$.

Since $y\ba \ge_{\bar y\d(\a)} \ba$, $sy\ba \ge_{\bar s \bar y \d(\a)} sy\d(s) \ba$. As $\ell(sy)>\ell(sy\d(s))$, $\ba, sy\d(s)\ba$ are on the same side of $sy\d H$. Moreover, the closure $sy\d(s) \bar \ba$ intersects with $sy\d H$. Therefore $sy\d(s)\ba \ge_{\bar s \bar y \d(\a)} \ba$.
\end{proof}

\begin{lemma}\label{alpha-jj}
Let $J, J' \subset \SS$ with $\d(J)=J$, $\d(J')=J'$. Let $y \in \tW$ and $\a \in \Sigma$. If there exist $w, w' \in W$ such that $w \i y \d(s_\a) \d(w) \in \tW_J$ and $(w') \i y \d(w') \in \tW_{J'}$, then $w \i(\a) \in \Sigma_J$ or $(w') \i(\a) \in \Sigma_{J'}$.
\end{lemma}

\begin{proof}
Let $V$ be the real vector space spanned by the coweights. Let $v_0, v'_0 \in V^\d$ be dominant coweights such that  for any $u \in W$, $u(v_0)=v_0$ (resp. $u(v_0')=v_0'$) if and only if $u \in W_J$ (resp. $u \in W_{J'}$). Set $v=w(v_0)$ and $v'=w'(v_0')$. Then $y \d(v')=y \d(w'v'_0)=w' (v'_0)=v'$ and $y \d(s_\a) \d(v)=y \d(s_\a) \d(w) (v_0)=w (v_0)=v$. Now \begin{align*}\<v'-s_\a(v),v'-s_\a(v)\> &=\<y \d(v'-s_\a(v)),y \d(v'-s_\a(v))\>\\ &=\<v'-v,v'-v\>.\end{align*} Hence $\<v', s_\a(v)\>=\<v', v\>$. If $w \i(\a) \notin \Sigma_J$, then $\<v, \a\> \neq 0$. So $\<v', \a\>=0$ and $(w') \i(\a) \in \Sigma_{J'}$.
\end{proof}

\begin{theorem}\label{compatible}
Let $y \in \tW$ such that property (NLO) holds for $y$. Let $s \in \tSS$.

(1) If $\ell(s y \d(s))=\ell(y)$, then property (NLO) holds for $s y \d(s)$;

(2) If $\ell(s y \d(s))=\ell(y)-2$, then property (NLO) holds for $s y \d(s)$ or $y \d(s)$.
\end{theorem}

\begin{proof}
Case 1: Assume that for any $J \subset \SS$ with $\d(J)=J$ and $w \in W$ such that $s y \d(s) \ba$ is a $(J, w, \d)$-alcove, $y \ba$ is also a $(J, \bar s w, \d)$-alcove. This in particular includes part (1) of the theorem.

So assume that $y$ satisfies property (NLO), and that $sy\d(s)$ is a $(J, w, \d)$-alcove. In this case, by assumption there is an element $b_J \in \bar s w \tW_J \d(\bar s w)\i$ satisfying conditions (1)-(3) in the definition of property (NLO) for $(y, J, \bar s w, \d)$.

Set $b'_J= s b_J \d(s) \in w \tW_J \d(w)\i$. Then $b'_J$ satisfies conditions (1)-(3) in the definition of property (NLO) for $(s y \d(s), J, w, \d)$.

Case 2: There exists $J \subset \SS$ with $\d(J)=J$ and $w \in W$ such that $s y \d(s) \ba$ is a $(J, w, \d)$-alcove, but $y \ba$ is not a $(J, \bar s w, \d)$-alcove. Hence $w\i(\a) \notin \Sigma_J$ by Lemma \ref{lem.2.1}, where $\a \in \Sigma$ such that $\bar s=s_\a$. We show that $y \d(s)$ satisfies property (NLO).

Assume $y\d(s) \ba$ is a $(J', w', \d)$-alcove for some $w' \in W$ and $J' \subset \SS$ with $\d(J')=J'$. By Lemma \ref{alpha-jj}, ${w'}\i(\a) \in \Sigma_{J'}$. Hence by Lemma \ref{lem.2.1}, $y \ba$ is a $(J', w', \d)$-alcove. Since $y$ satisfies property (NLO), there exists $b_{J'} \in w' \tW_{J'} \d({w'})\i$ such that $\kappa_G(b)=\kappa_G(b_{J'})$, $\nu_{b}=\nu_{b_{J'}}$ and \begin{align*}\kappa_{J'}({w'}\i b_{J'} \d(w'))&=\kappa_{J'}({w'}\i y \d(w'))\\ &=\kappa_{J'}({w'}\i y \d(w') \d({w'}\i s w'))=\kappa_{J'}({w'}\i y\d(s) w'),\end{align*} where the second equality follows from ${w'}\i(\a) \in \Sigma_{J'}$. Hence property (NLO) holds for $y\d(s)$.
\end{proof}

Next we consider the case of minimal length elements:

\begin{prop}\label{basic-step}
If $y$ is a minimal length element in its $\d$-twisted conjugacy class and $y$ satisfies property (NLO), then $X_y(b) \neq \emptyset$.
\end{prop}
\begin{proof}
It suffices to show that $y$ and $b$ are in the same $\s$-conjugacy class. By Kottwitz \cite{kottwitz-isoI} and \cite{kottwitz-isoII}, this is equivalent to show that $\bar \nu_y=\nu_b$ and $\kappa_G(y)=\kappa_G(b)$. Since $y\ba$ is automatically a $(\SS, 1, \d)$-alcove, by our assumption $\kappa_G(y)=\kappa_G(b)$.

By \cite[Theorem 2.10]{HN}, there exists a minimal length element $y'$ in the $\d$-twisted conjugacy class containing $y$ such that $\bar \ba \cap V_{y'} \neq \emptyset$. Here $V_{y'}=\{v \in V; y'\d(v)=v+\nu_{y'}\}$. We may then assume that $y=y'$ (use Theorem~\ref{compatible} (1)).

Let $w \in W$ such that $\bar \nu=w\i(\nu_y)$ is dominant. Then $w(\bar \nu)=\nu_y=\bar y \d (\nu_y)=\bar y \d(w) \d(\bar \nu)$. In other words, $\bar \nu=w\i \bar y \d(w) \d(\bar \nu)$. Since $\d(\bar \nu)$ is the unique dominant coweight in the $W$-orbit of $\bar \nu$, we have $\bar \nu=\d(\bar \nu)$. Set $J=\{s \in \SS; s(\bar \nu)=\bar \nu\}$. Then $\d(J)=J$ and $w\i \bar y \d(w) \in W_J$. 

For any $\b \in \Sigma$ with $w\i(\b) \in \Sigma^+-\Sigma^+_J$, $\<\nu_y, \b\>=\<\bar \nu, w\i(\b)\> > 0$. Hence $y \ba \ge_{\b} \ba$ as $\bar \ba$ intersects with $V_{y}$. Therefore $y \ba$ is a $(J, w, \d)$-alcove.

Since $y$ satisfies property (NLO), there exists $b_J \in w \tW_J \d(w)\i$ such that $\nu_{b_{J}} = \nu_b$ and $\kappa_{J}(w\i b_{J} \d(w)) = \kappa_{J}(w\i y\d(w))$. If we write $w\i y \d(w)=\e^\l u$ and $w\i b_J \d(w)=\e^{\l'} u'$, $u, u' \in W_J$, then for a suitable coweight $\th$, 
\[
\l-\l'+\th-\d(\th) \in Q^\vee_J.
\] 
Thus $\nu_\l-\nu_{\l'} \in Q^\vee_J \otimes_\ZZ \QQ$. By Section \ref{newton} (1), $\nu_{w\i y \d(w)}-\nu_{w\i b_J \d(w)} \in Q^\vee_J \otimes_\ZZ \QQ$. Since $\nu_{w\i y \d(w)}=\bar \nu$ is orthogonal to all the roots in $J$ and  $\nu_{w\i b_J \d(w)}=\nu_b$ is central, by Section \ref{newton} (2), $\bar \nu_y=\nu_{w\i y \d(w)}=\nu_b$. 
\end{proof}

Altogether, we have now proved:

\begin{theorem}\label{thm-nonemptiness}
Let $x\in \tW$. If $x$ satisfies property (NLO), then $X_x(b) \ne \emptyset$.
\end{theorem}
\begin{rem}
For split groups, this was conjectured in \cite{GHKR2}, Conjecture 9.4.2.
\end{rem}
\begin{proof} Since $x$ satisfies property (NLO), there exists $y \in D_{x, \d}$ also satisfies property (NLO). By Proposition \ref{basic-step}, $X_y(b) \neq \emptyset$. Hence by Theorem \ref{dxd}, $X_x(b) \neq \emptyset$.
\end{proof}

Now combining Theorem \ref{thm-nonemptiness} with Proposition \ref{reuman2}, we have 
\begin{prop}\label{reuman4}
Let $b \in G(L)$ be basic and $x \in \tW$ lie in the shrunken Weyl chambers such that $\k_G(b)=\k_G(x)$. If $\eta_\d(x)\in W-\bigcup_{J\subsetneq \SS, \d(J)=J} W_J$, then $X_x(b) \neq \emptyset$. 
\end{prop}

\end{document}